\documentclass[12pt, a4paper]{elsarticle}
\usepackage{amsthm,amsmath,amssymb,extsizes,easymat,graphicx,hyperref,easybmat}

\usepackage[all]{xy}

\usepackage[active]{srcltx}
\usepackage{MnSymbol}
\usepackage{rotating}
\usepackage{mathrsfs,color}
\DeclareMathOperator{\rank}{rank}
\DeclareMathOperator{\cod}{cod}
\DeclareMathOperator{\diag}{diag}
\DeclareMathOperator{\orb}{O}

\DeclareMathOperator{\tsp}{T}

\DeclareMathOperator{\GSYL}{GSYL}

\DeclareMathOperator{\rev}{rev}
\DeclareMathOperator{\POL}{POL}
\DeclareMathOperator{\PEN}{PENCIL}
\DeclareMathOperator{\skewPEN}{SSPENCIL}

\DeclareMathOperator{\Inf}{INF}

\usepackage{tcolorbox}
\tcbuselibrary{skins}
\tcbuselibrary{breakable}
\tcbset{shield externalize}

\newcommand{\sdotsss}%
{\text{\raisebox{-2.2pt}{$\cdot\,$}%
 \raisebox{1.7pt}{$\cdot$}%
\raisebox{5.6pt}{$\,\cdot$}}}

\renewcommand{\le}{\leqslant}
\renewcommand{\ge}{\geqslant}

\newtheorem{theorem}{Theorem}[section]

\newtheorem{lemma}[theorem]{Lemma}

\newtheorem{definition}[theorem]{Definition}

\newtheorem{example}[theorem]{Example}

\newcommand{\hide}[1]{}

\begin{document}

\title{Even grade generic skew-symmetric matrix polynomials with bounded rank}

%\title{Generic (skew-)symmetric matrix polynomials with fixed rank and fixed even grade\tnoteref{t1}}

\date{}

\author[cm]{Fernando De Ter\'{a}n}
\ead{fteran@math.uc3m.es}

\author[um]{Andrii Dmytryshyn}
\ead{andrii.dmytryshyn@oru.se}

\author[cm]{Froil\'an M. Dopico}
\ead{dopico@math.uc3m.es}

\address[cm]{Departamento de Matem\'aticas, Universidad Carlos III de Madrid, Avenida de la Universidad 30, 28911, Legan\'es, Spain.}

\address[um]{School of Science and Technology, \"{O}rebro University, 701 82, \"{O}rebro, Sweden.}

\begin{abstract}
We show that the set of $m \times m$ complex skew-symmetric matrix polynomials of even grade $d$, i.e., of degree at most $d$, and (normal) rank at most $2r$ is the closure of the single set of matrix polynomials with certain, explicitly described, complete eigenstructure. This complete eigenstructure corresponds to the most generic $m \times m$ complex skew-symmetric matrix polynomials of even grade $d$ and rank at most $2r$. The analogous problem for the case of skew-symmetric matrix polynomials of odd grade is solved in \cite{DmDo18}.
\end{abstract}


\maketitle

\section{Introduction}
During the recent years, the problem of determining the generic eigenstructures for sets of matrix pencils and matrix polynomials has been a subject of research interest, in particular, for the sets of matrix pencils and matrix polynomials with fixed grade and bounded rank. The description of such sets as a union of closures of certain ``generic sets'' of pencils and matrix polynomials is given in terms of their eigenstructures for general matrix pencils \cite{DeDo08} and polynomials \cite{DmDo17}, as well as for matrix pencils that are skew-symmetric \cite{DmDo18}, symmetric \cite{DeDD20a}, $T$-palindromic { and $T$-alternating \cite{DeTe18}, and} Hermitian \cite{DeDD22}. Moreover, in the case of odd grade the generic sets are also derived for skew-symmetric \cite{DmDo18} and symmetric \cite{DeDD20b} matrix polynomials. In this paper, we tackle the even-grade-case for skew-symmetric matrix polynomials. In Table \ref{table1} we present a summary of the contributions mentioned above.

The reason for the even-grade-case being frequently omitted when structured matrix polynomials are studied, e.g., \cite{DeDD20b,Dmyt16,DmDo18}, is the lack of structured linearization templates \cite{DeDM14,MMMM10,MMMM13}. In this paper, we resolve this issue by adding an extra term to  a given matrix polynomial  of even grade (thus making it of odd grade) and tracking how  this addition affects the eigenstructure of the polynomial. The main result is that the set of skew-symmetric matrix polynomials of even grade and bounded rank is the closure of a single set with a certain complete eigenstructure { which is completely analogous to the one for the odd-grade-case described in \cite{DmDo18}}.

\begin{table}
\label{table1}
\centering
{\scriptsize
%\vskip-0.2cm \hspace*{-3.5mm}
  \begin{tabular}{|c|c|c|c|c|}
    \hline
         &  & \\
     & {\bf Pencils ($d =1$) } & {\bf Polynomials ($d >1$) } \\
              &  & \\ \hline
                &&  \\[-0.8em]
    {\bf General} & \begin{tabular}{l} De Ter\'an and Dopico, \\2008, \cite{DeDo08} \\  \# = $r + 1$ \\ De Ter\'an, Dopico\\ and Landsberg, \\ 2017, \cite{DeDL17} \end{tabular} & \begin{tabular}{l} Dmytryshyn and Dopico, \\2017, \cite{DmDo17} \\ \# = $rd + 1$ \end{tabular} \\ [0.2em] \hline
      &&  \\[-0.8em]
    {\bf Skew-Symmetric} & \begin{tabular}{l} Dmytryshyn and Dopico, \\ 2018, \cite{DmDo18} \\ \# = $1$ \end{tabular} & \begin{tabular}{l} Dmytryshyn and Dopico,\\ 2018, \cite{DmDo18} \\ odd $d$\\ \# =  $1$ \\  \\  {\bf even $d$, this paper} \end{tabular}  \\ [0.2em] \hline
      &&  \\[-0.8em]
     {\bf  Symmetric} & \begin{tabular}{l} De Ter\'an, Dmytryshyn\\ and Dopico, \\ 2020, \cite{DeDD20a} \\  \# = $\left\lfloor r/2 \right\rfloor + 1$ \end{tabular} &  \begin{tabular}{l} De Ter\'an, Dmytryshyn\\ and Dopico, \\ 2020, \cite{DeDD20b} \\odd $d$\\ \# = $\left\lfloor r \, d /2 \right\rfloor + 1$ \\   \end{tabular}  \\ [0.2em] \hline
       &&  \\[-0.8em]
    \begin{tabular}{l}  {\bf  T-(anti)palindromic} \\ {\bf(or T-even and odd)}  \end{tabular}  &   \begin{tabular}{l}  De Ter\'an, 2018, \cite{DeTe18}  \\ { \# = $1$} \end{tabular}& { open} \\ [0.2em]\hline
  &&  \\[-0.7em]
   {\bf  (skew-)Hermitian} & \begin{tabular}{l} De Ter\'an, Dmytryshyn\\ and Dopico, \\ 2022, \cite{DeDD22} \\  \# = $\left\lfloor (r+3)/2 \right\rfloor \left(\left\lfloor r/2 \right\rfloor + 1\right)$ \end{tabular} & { open}  \\ [0.2em] \hline
  % {\bf  port-Hamiltonian} & { open} & {open}& { open} & {open} \\ \hline
   %{\bf  descriptor} & { open} & {open} & { open} & { open} \\ \hline
  \end{tabular}}
    \caption{Solved and open problems about generic eigenstructures of sets of general and structured matrix polynomials with bounded rank and degree. Notation: $r$ is the rank, $d$ is the grade, and \# is the number of generic eigenstructures.}
    \label{tab1}
\end{table}

%Skew-symmetric pencils and matrix polynomials appear in a number of interesting applications, e.g., analysis of passive velocity field controllers \cite{LiHo99}, bi-Hamiltonian systems \cite{Olve91}, multisymplectic PDEs \cite{BrRe01}, investigation of a product of two skew-symmetric matrices and particular symmetric factorizations of skew-symmetric rational matrices \cite{MMMM13}.

One reason to study generic sets is the possible applications in the investigation of the effect of low rank perturbations on the spectral information of pencils and matrix polynomials \cite{Batzke14,Batzke15,Batzke16, DeDo07,DeDo09,DeDo16,DeDoMo08,DeMM22, HMP19, HMP23, MMW17}. We also refer the reader to the introductions of the papers cited in Table \ref{tab1} for more references and background information. Determining the generic eigenstructures of particular sets of structured matrix pencils and matrix polynomials (in particular, those with bounded rank) is also useful for obtaining the stratification of structured matrix pencils and polynomials. Describing the stratification of matrix pencils and polynomials has been considered in some references by different authors, for several sets of matrix pencils and polynomials, like \cite{EdEK99} for general matrix pencils, \cite{ElJK09} for controllability and observability pairs, \cite{DmJK17} for system pencils associated with state-space systems, \cite{JoKV13} for full rank matrix polynomials, and \cite{DJKV19} for general matrix polynomials, as well as for skew-symmetric matrix pencils \cite{DmKa14} and odd-grade polynomials \cite{Dmyt15}.

%Moreover, from a more theoretical perspective, the study of the sets of matrix pencils and of matrix polynomials with fixed grade and fixed rank generalizes  classical studies \cite{Wate84} on the algebraic structure of the set of $n\times n$ singular pencils, i.e., those whose rank is at most $n-1$.

To facilitate the reading of this paper and its comparison with previous results, we keep its structure and style as close as possible to the paper \cite{DmDo18} that covers the odd-grade-case. In this respect, we also warn the reader that several repetitions of background concepts and auxiliary results appearing in \cite{DmDo18} are unavoidable to keep this paper self-contained. Thus, Section \ref{pencils} includes some basic results on general and skew-symmetric pencils and { also a new} description of the set of skew-symmetric matrix pencils with rank { at most} $2 w$ and fixed number of { canonical} blocks associated  with the infinite eigenvalue. Section \ref{sect.polys} presents a number of preliminary { known} results on skew-symmetric matrix polynomials { and some new auxiliary results on orbits of skew-symmetric matrix polynomials}. In Section \ref{sec.main} we prove our main result for generic skew-symmetric matrix polynomials of fixed even grade and fixed rank, i.e., Theorem \ref{mainth}. Moreover, Theorem \ref{mainth} is combined with the corresponding result for odd-grade skew-symmetric matrix polynomials obtained in \cite{DmDo18} for presenting in Theorem \ref{anydth} the general result, which is independent of the parity of the grade. Finally, Section \ref{sec.codim} includes the computation of the codimension of the orbit of the generic skew-symmetric matrix polynomials of fixed even grade and bounded rank, which requires a different approach to the  one used to obtain the corresponding result for the odd-grade-case considered in \cite{DmDo18}. The conclusions and future research on generic eigenstructures of matrix polynomials are discussed in Section \ref{sect.conclusions}.

\section{Skew-symmetric matrix pencils}
\label{pencils} In this paper we consider { matrix pencils and matrix polynomials whose matrix coefficients have complex entries} { namely, a matrix polynomial is of the form $P(\lambda)=\sum_{i=0}^d\lambda^{i}A_i$, with $A_i\in\mathbb C^{n\times n}$, and a matrix pencil is a matrix polynomial with $d=1$. For matrix pencils we use calligraphic letters}. { For brevity, the dependence on $\lambda$ in $P(\lambda)$ will be often omitted, specially in the proofs, and we will use just $P$ for denoting a matrix polynomial, and similarly for matrix pencils.}

The rank of a matrix polynomial  $P$, denoted by $\rank P$, is  the largest size of a non-identically zero minor  (see \cite[Ch. XII, \S3]{Gant59} for matrix pencils). 

\subsection{Preliminaries}
In this section we recall the Kronecker canonical form (KCF) of general matrix pencils and the canonical form of skew-symmetric matrix pencils under congruence. Define
$\overline{\mathbb C} := \mathbb C \cup  \infty$.
For each $k=1,2, \ldots $, { and for each $\mu \in \mathbb{C}$} define the $k\times k$
matrices
\begin{equation*}
J_k(\mu):=\begin{bmatrix}
\mu&1&&\\
&\mu&\ddots&\\
&&\ddots&1\\
&&&\mu
\end{bmatrix},\qquad
I_k:=\begin{bmatrix}
1&&&\\
&1&&\\
&&\ddots&\\
&&&1
\end{bmatrix},
\end{equation*}
and for each $k=0,1, \ldots $, define the $k\times
(k+1)$ matrices
\begin{equation*}
F_k :=
\begin{bmatrix}
0&1&&\\
&\ddots&\ddots&\\
&&0&1\\
\end{bmatrix}, \qquad
G_k :=
\begin{bmatrix}
1&0&&\\
&\ddots&\ddots&\\
&&1&0\\
\end{bmatrix}.
\end{equation*}
All non-specified entries of $J_k(\mu), I_k, F_k,$ and $G_k$ are zeros.

Two $m \times n$ matrix pencils $\lambda A - B$ and $\lambda C - D$ are called {\em strictly equivalent} if and only if there are non-singular matrices $Q$ and $R$ such that $Q^{-1}AR =C$ and $Q^{-1}BR=D$. We also define the {\it orbit} of $\lambda A - B$ under the
action of the group $GL_m(\mathbb C) \times GL_n(\mathbb C)$ on the space of all matrix pencils by strict equivalence as follows:
\[
\orb^e (\lambda A - B) = \{Q^{-1} (\lambda A - B) R \ : \ Q \in GL_m(\mathbb C), R \in GL_n(\mathbb C)\}.
\]
    
Theorem \ref{kron} introduces the Kronecker Canonical Form (KCF) of  matrix pencils.

\begin{theorem}{\rm \cite[Ch. XII, Sect. 4]{Gant59}}\label{kron}
Each $m \times n$ matrix pencil $\lambda A - B$ is strictly equivalent
to a direct sum, uniquely determined up
to permutation of summands, of pencils of the form
\begin{align*}
{\cal E}_k(\mu)&:=\lambda I_k - J_k(\mu), \text { with} \mu \in \mathbb C, \quad {\cal E}_k(\infty):=\lambda J_k(0) - I_k, \\
{\cal L}_k&:=\lambda G_k - F_k, \quad \text{ and } \quad {\cal L}_k^T:=\lambda G_k^T - F^T_k.
\end{align*}
This direct sum is called the KCF of $\lambda A -B$.
\end{theorem}
\noindent  The blocks ${\cal E}_k(\mu)$ and ${\cal E}_k(\infty)$  are associated to the finite and infinite eigenvalues, respectively, { and all together} form the {\em regular part} of $\lambda A - B$. The blocks ${\cal L}_k$ and ${\cal L}_k^T$  are associated to the  right and  left minimal indices, respectively  (see page \pageref{mi}),  and all together form the {\em singular part} of $\lambda A - B$. The number of blocks ${\cal L}_k$ (respectively, ${\cal L}_k^T$) in the KCF of $\lambda A - B$ is equal to the dimension of the right (respectively, left) rational null-space of $\lambda A - B$.

%The orbit of $\lambda A - B$ is a {   differential} manifold in the complex $2mn$ dimensional space.

A pencil $\lambda A - B$ is {\em skew symmetric} if and only if $(\lambda A - B)^T = -(\lambda A - B)$.
An $n \times n$ matrix pencil $\lambda A - B$ is called {\it congruent} to $\lambda C - D$ if there is a non-singular matrix $S$ such that $S^{T}AS =C$ and $S^{T}BS=D$.
We also define {   the} {\it congruence orbit} of $\lambda A - B$ under the action of the group $GL_n(\mathbb C)$ on the space of all skew-symmetric matrix pencils by congruence as follows:
\[
\orb^c (\lambda A - B) = \{S^{T} (\lambda A - B) S \ : \ S \in GL_n(\mathbb C)\}.
\]
%The congruence orbit of $\lambda A - B$ is a {   differential} manifold in the complex $n^2-n$ dimensional space of skew-symmetric pencils.

In Theorem \ref{lkh} we recall the canonical form under congruence of skew-symmetric matrix pencils, so called {\em skew-symmetric KCF}.
\begin{theorem}{\rm \cite{Thom91} }\label{lkh}
Each skew-symmetric $n \times n$ matrix pencil $\lambda A - B$ is congruent
to a direct sum,
uniquely determined up
to permutation of
summands, of pencils of
the form
\begin{align*}
{\cal H}_{h}(\mu)&:=
\lambda
\begin{bmatrix}0&I_h\\
-I_h &0
\end{bmatrix} -
\begin{bmatrix}0&J_h(\mu)\\
-J_h(\mu)^T &0
\end{bmatrix}
,\quad \mu \in\mathbb C,\\
{\cal K}_k&:=
\lambda
\begin{bmatrix}0&J_k(0)\\
-J_k(0)^T &0
\end{bmatrix} -
\begin{bmatrix}0&I_k\\
-I_k&0
\end{bmatrix},\\
{\cal M}_m&:=
\lambda
\begin{bmatrix}0&G_m\\
-G_m^T&0
\end{bmatrix} -
\begin{bmatrix}0&F_m\\
-F_m^T &0
\end{bmatrix}.
\end{align*}
\end{theorem}
\noindent Notably the block ${\cal M}_0$ is the $1 \times 1$ zero matrix pencil.
Similarly to the KCF, the blocks ${\cal H}_{h}(\mu)$ and ${\cal K}_k$ correspond to the finite and infinite eigenvalues, respectively, and all together form the {\em regular part} of $\lambda A - B$. The blocks ${\cal M}_m$ correspond to the right (column) and left (row) minimal indices, which are equal in the case of skew-symmetric matrix pencils, and form the {\em singular part} of $\lambda A - B$. Theorem \ref{lkh} also shows that skew-symmetric matrix pencils always have even rank.

 For brevity, we will refer to the number of $\cal K$-blocks or $\cal M$-blocks of a skew-symmetric matrix pencil to mean the number of blocks of the form ${\cal K}_k$ or ${\cal M}_k$, respectively, in the skew-symmetric KCF of the pencil.

\subsection{Generic skew-symmetric matrix pencils with bounded rank and fixed number of blocks corresponding to the  infinite eigenvalue.} \label{sec.main1}

In this section we find the most generic { skew-symmetric matrix pencils with rank bounded by a fixed value and with exactly $r$ ${\cal K}$-blocks in the skew-symmetric KCF}. First, we recall some definitions and results.

By $\PEN_{m \times n}$ we denote the space of all $m\times n$ matrix pencils. A distance in $\PEN_{m \times n}$ can be defined with the Frobenius norm of complex matrices \cite{Highambook} as $d(\lambda A - B, \lambda C - D) := \sqrt{\|A-C\|_F^2 + \|B-D\|_F^2}$, which makes $\PEN_{m \times n}$ { to be} a metric space. This metric allows us to consider closures of subsets of $\PEN_{m \times n}$, in particular, closures of orbits by strict equivalence, denoted by $\overline{\orb^e}(\lambda A - B)$. Using these concepts, Theorem \ref{erelations}, see \cite{Bole98, EdEK99}, describes all the possible changes in the KCF of a general unstructured matrix pencil under arbitrarily small perturbations.
If in the KCF of a matrix pencil the blocks ${\cal X}$ are changed to the blocks ${\cal Y}$ (${\cal X}$ and ${\cal Y}$ of the same size) we write ${\cal X} \rightsquigarrow {\cal Y}$.

\begin{theorem}{\rm \cite{Bole98}}\label{erelations}
Let ${\cal P}_1$ and ${\cal P}_2$ be two matrix pencils in KCF. Then,  $\overline{\orb^e}({\cal P}_1) \supset \orb^e ({\cal P}_2)$
if and only if ${\cal P}_1$ can be obtained from ${\cal P}_2$ { by} changing canonical blocks of ${\cal P}_2$ { after} applying a sequence of rules, { that can be} of the { following} six types:
\begin{enumerate}
\item ${\cal L}_{j-1} \oplus {\cal L}_{k+1} \rightsquigarrow {\cal L}_j \oplus {\cal L}_k$, $1\le j \le k;$
\item ${\cal L}_{j-1}^T \oplus {\cal L}_{k+1}^T \rightsquigarrow {\cal L}_j^T \oplus {\cal L}_k^T$, $1\le j \le k;$
\item ${\cal L}_{j} \oplus {\cal E}_{k+1}(\mu) \rightsquigarrow {\cal L}_{j+1} \oplus {\cal E}_k(\mu)$, $j,k=0,1,2, \dots$ and $\mu \in \overline{\mathbb C};$
\item ${\cal L}_{j}^T \oplus {\cal E}_{k+1}(\mu) \rightsquigarrow {\cal L}_{j+1}^T \oplus {\cal E}_k(\mu)$, $j,k=0,1,2, \dots$ and $\mu \in \overline{\mathbb C};$
\item ${\cal E}_{j}(\mu) \oplus {\cal E}_{k}(\mu) \rightsquigarrow {\cal E}_{j-1}(\mu) \oplus {\cal E}_{k+1}(\mu)$, $1\le j \le k$ and $\mu \in \overline{\mathbb C};$
\item ${\cal L}_{p} \oplus {\cal L}_{q}^T \rightsquigarrow \bigoplus_{i=1}^t{\cal E}_{k_i}(\mu_i)$, if $p+q+1= \sum_{i=1}^t k_i$ and $\mu_i \neq \mu_{i'}$ for $i \neq i', \mu_i \in \overline{\mathbb C}.$
\end{enumerate}
Observe that in the rules above any block ${\cal E}_0 (\mu)$ should be understood as the empty matrix.
\end{theorem}

The vector space of skew-symmetric matrix pencils of size $n\times n$ is denoted by $\PEN_{n\times n}^{ss}$. A distance in $\PEN_{n\times n}^{ss}$ is defined as in $\PEN_{n\times n}$ and, thus, the topology considered in $\PEN_{n\times n}^{ss}$ is the induced topology from $\PEN_{n\times n}$. With this topology, we also consider closures of subsets in $\PEN_{n\times n}^{ss}$ and, in particular, closures of orbits by congruence of skew-symmetric matrix pencils, which are denoted by $\overline{\orb^c} (\lambda A - B)$.

We often use expressions as ``the pencil ${\cal P}_1$ is more generic than the pencil ${\cal P}_2$'', whose meaning is that $\overline{\orb^e} ({\cal P}_1) \supset \orb^e ({\cal P}_2)$ or $\overline{\orb^c} ({\cal P}_1) \supset \orb^c ({\cal P}_2)$, depending on the context. In this language, the most generic skew-symmetric pencil with { rank at most $2w$ and with exactly $r$ ${\cal K}$-blocks associated to the infinite eigenvalue in the skew-symmetric KCF}, is a pencil such that the closure of its congruence orbit includes the congruence orbit of any other skew-symmetric pencil with { rank at most $2w$} and with exactly { $r$ ${\cal K}$-blocks}.  Define { $\Inf_r \subset \PEN_{n\times n}^{ss}$} as the set of all $n \times n$ skew-symmetric matrix pencils  having exactly $r$ { ${\cal K}$-blocks}.

\begin{theorem} \label{th:improvedDeDo}
Let $n,w$, and $r$ be integers such that $n \geq 2$, $2\leq 2w \leq n-1$ and  $r \leq w$. The set of $n \times n$ complex skew-symmetric matrix pencils { with rank at most $2w$ and with exactly $r$ ${\cal K}$-blocks corresponding to the infinite eigenvalue in its skew-symmetric KCF}, is a closed subset of $\Inf_r$ equal to $\overline{\orb^c}({\cal W}) \cap \Inf_r$, where
\begin{equation}\label{max}
{\cal W} =\diag(\underbrace{{\cal M}_{\alpha+1},\dots,{\cal M}_{\alpha+1}}_{s},
\underbrace{{\cal M}_{\alpha},\dots,{\cal M}_{\alpha}}_{n-2w-s}, \underbrace{{\cal K}_{1},\dots,{\cal K}_{1}}_{r})\,
\end{equation}
with $\alpha= \lfloor (w-r)/(n-2w) \rfloor$ and
$s\equiv (w-r) \, \mathrm{mod}\, (n-2w)$.
\hide{
Let us define, in the set of $n\times n$ complex skew-symmetric matrix pencils with
rank $2r_1$, the following skew-symmetric canonical form:
\begin{equation}\label{max}
{\cal W} (\lambda)=\diag(\underbrace{M_{\alpha+1},\hdots,M_{\alpha+1}}_{s},
\underbrace{M_{\alpha},\hdots,M_{\alpha}}_{n-2r_1-s})\,
\end{equation}
where $\alpha= \lfloor r_1/(n-2r_1) \rfloor$ and
$s= r_1 \, \mathrm{mod}\, (n-2r_1)$.
Then,
\begin{enumerate}
\item[\rm (i)] $\overline{\orb^c}({\cal M}) \supseteq\overline{\orb^c}({\cal X})$ for every $n_1\times n_1$ skew-symmetric matrix pencil ${\cal X}(\lambda)$ with rank at most $2r_1$.
\item[\rm (ii)] The set of $n_1 \times n_1$ complex matrix pencils with rank
at most $2r_1$ is a closed  subset of $\skewPEN_{n_1 \times n_1}$ equal to $\overline{\orb^e}({\cal M} )$.
\end{enumerate}
}
\end{theorem}
\begin{proof}
{ We follow a strategy similar to the one of the proof of \cite[Thm.~3.1]{DmDo18}, but taking into account the presence of $r$ ${\cal K}$-blocks, which were not present in \cite[Thm.~3.1]{DmDo18}}. More precisely, taking into account \cite[Lem. 3.8]{DmKa14} (or the stronger result \cite[Thm.~3.1]{DmKa14}), in this proof we work with the KCF rather than with the skew-symmetric KCF under congruence of Theorem \ref{lkh}, but we always apply rules from Theorem \ref{erelations} in pairs, such that the corresponding change of the skew-symmetric KCF { preserves transparently the skew-symmetry}.

{ Note that the most generic skew-symmetric pencil with rank at most $2w$ and exactly $r$ ${\cal K}$ blocks has rank exactly $2w$, because otherwise there would be more than $n-2w$ pairs of equal left and right minimal indices (equivalently, in the skew-symmetric KCF there would be more than $n-2w$ ${\cal M}$-blocks, since the number of left/right singular blocks is equal to the dimension of the rational left/right null-space of the pencil). Moreover, due to the even parity of the rank, there would be an excess of at least two pairs of equal left and right minimal indices. In that case, the rule 6 in Theorem \ref{erelations} can be applied twice to obtain a more generic skew-symmetric matrix pencil with rank larger in two units and with exactly $r$ ${\cal K}$-blocks. Therefore, we focus, in the rest of the proof, on  $n \times n$ skew-symmetric matrix pencils with rank exactly $2 w$ and with exactly $r$ ${\cal K}$-blocks. Each of such pencils} has the following KCF:
 \[
 \begin{split}
\diag({\cal L}_{\gamma_1},\dots,{\cal L}_{\gamma_{n-2w}}, &{\cal L}^T_{\gamma_1},\dots,{\cal L}^T_{\gamma_{n-2w}}, {\cal J}, {\cal J}, \\
&{\cal E}_{k_1}(\infty),\dots,{\cal E}_{k_r}(\infty),
{\cal E}_{k_1}(\infty),\dots,{\cal E}_{k_r}(\infty)),
\end{split}
\]
since the left and right singular blocks are paired up according { to Theorem \ref{lkh} and the} blocks ${\cal J}$ of the regular part associated with  the finite  eigenvalues are paired up, as  well as the blocks that correspond to the infinite eigenvalue.

Note that the sizes of { the $2r$ blocks} that correspond to the infinite eigenvalue, { i.e. ${\cal E}_{k_i}(\infty)$,} in the KCF of a generic skew-symmetric matrix pencil of rank equal to $2w$ and with exactly $r$ { ${\cal K}$-blocks} can not be larger than 1 (otherwise rules 3 and 4 from Theorem \ref{erelations} can be applied to obtain a more generic skew-symmetric pencil and to decrease the sizes of the $2r$ blocks ${\cal E}_{k_i}(\infty)$ to $1$). The remaining part of the proof repeats the proof of \cite[Thm.~3.1]{DmDo18}  with the difference that \cite[Thm. 2.7]{DmDo17} is now used for determining the most generic $(w-r) \times (n-r-w)$ matrix pencil and \cite[Thm. 2.8]{DmDo17} for determining the most generic $(n-r-w) \times (w-r)$ matrix pencil. The details can be found in \cite{DmDo18},  but the proof reduces to apply rules 3 and 4 from Theorem \ref{erelations} for every couple of identical single Jordan blocks associated with a finite eigenvalue (using a left and a right singular block ${\cal L}_j$ and ${\cal L}_j^T$, which are coupled) to eliminate these Jordan blocks by shrinking the size to $0$, and then applying rules 1 and 2 from Theorem \ref{erelations} to each couple of blocks ${\cal L}_{j-1}\oplus{\cal L}_{k+1}$ and ${\cal L}_{j-1}^T\oplus{\cal L}_{k+1}^T$.
\end{proof}

\section{{ Auxiliary results} on skew-symmetric matrix polynomials}
\label{sect.polys}
\subsection{Complete eigenstructure of skew-symmetric matrix polynomials}
\label{sect.prempolys}

We consider skew-symmetric $m\times m$ matrix polynomials $P(\lambda)$ of grade $d$, i.e., of degree less than or equal to $d$, over~$\mathbb C$:
\begin{equation*}
P(\lambda) = \lambda^{d}A_{d} + \dots +  \lambda A_1 + A_0,
\quad A_i^T=-A_i, \ A_i \in \mathbb C^{m \times m} \ \text{for } i=0, \dots, d.
\end{equation*}
{ Note that we do not require that the leading coefficient, $A_d$, of $P(\lambda)$ is nonzero. As usual, the degree of $P(\lambda)$, denoted as $\deg(P)$, is the largest index $k$ such that $A_k \ne 0$.}
 We denote the vector space of $m \times m$ skew-symmetric matrix polynomials of grade $d$ by $\POL_{d, m\times m}^{ss}$.
 We write $\POL$ instead of $\POL_{d, m\times m}^{ss}$,  if there is no risk of confusion. Note that $\POL_{1, n\times n}^{ss} = \PEN_{n\times n}^{ss}$. As in the case of pencils, see Sections \ref{pencils}, by using the Frobenius matrix norm of complex matrices \cite{Highambook}, a distance in $\POL_{d, m\times m}^{ss}$ is defined as $d(P,P') = \left( \sum_{i=0}^d || A_i - A'_i ||_F^2 \right)^{\frac{1}{2}}$, where $P(\lambda) = \sum_{i=0}^d \lambda^i A_i$ and $P'(\lambda) = \sum_{i=0}^d \lambda^i A'_i$, making $\POL_{d, m\times m}^{ss}$ a metric space  with the { Euclidean topology induced by this distance}. For convenience, we define the Frobenius norm of the matrix polynomial $P$ as $||P(\lambda)||_F = \left( \sum_{i=0}^d || A_i ||_F^2 \right)^{\frac{1}{2}}$. An arbitrarily small pencil or matrix polynomial is a pencil or matrix polynomial with arbitrarily small Frobenius norm.

Two matrix polynomials $P(\lambda)$ and $Q(\lambda)$ are called {\it unimodularly congruent} if $F(\lambda)^T P(\lambda) F(\lambda)=Q(\lambda)$ for some unimodular matrix polynomial $F(\lambda)$ (i.e. $\det F(\lambda) \in \mathbb C \backslash \{0\}$), see also \cite{MMMM13}.
%In the following theorem we recall the canonical form for skew-symmetric matrix polynomials under unimodular congruence.
\begin{theorem}{\rm \cite{MMMM13}} \label{tsmiths}
Let $P(\lambda)$ be a skew-symmetric $m\times m$ matrix polynomial. Then there exist $r \in \mathbb N$ with $2r \le m$ and a unimodular matrix polynomial $F(\lambda)$ such that
\begin{equation*}
F(\lambda)^T P(\lambda) F(\lambda) =
\begin{bmatrix}
0 & g_1(\lambda)\\
-g_1(\lambda) & 0
\end{bmatrix} \oplus
\dots
\oplus
\begin{bmatrix}
0 & g_r(\lambda)\\
-g_r(\lambda) & 0
\end{bmatrix} \oplus
0_{m-2r}=:S(\lambda),
\end{equation*}
where $g_j$ is a monic polynomial, for $j=1, \dots, r$, and $g_j(\lambda)$ divides $g_{j+1}(\lambda)$, for $j=1, \dots, r-1$. Moreover, the canonical form $S(\lambda)$ is unique.
\end{theorem}
Similarly to the unimodular congruence, two matrix polynomials $P(\lambda)$ and $Q(\lambda)$ are called {\it unimodularly equivalent} if $U(\lambda) P(\lambda) V(\lambda)=Q(\lambda)$ for some unimodular matrix polynomials \ $U(\lambda)$ \ and \ $V(\lambda)$ \ (i.e. $\det U(\lambda), \det V(\lambda) \in \mathbb C \backslash \{0\}$). Notably, the canonical form in Theorem \ref{tsmiths} is the skew-symmetric version of the well-known Smith form for matrix polynomials under unimodular equivalence \cite{Gant59,MMMM13}.

Note that the (normal) rank of the skew-symmetric matrix polynomial $P(\lambda)$ in Theorem \ref{tsmiths} is equal to the nonnegative integer $2r$, i.e., it is always an even number. The monic scalar polynomials $g_1(\lambda),\dots,g_r(\lambda)$ in Theorem \ref{tsmiths} are called the {\it invariant polynomials} of $P(\lambda)$, and, for any $\alpha \in \mathbb{C}$, each of them can be uniquely factored as
$$
g_j(\lambda) = (\lambda - \alpha)^{\sigma_j} p_j (\lambda), \quad \mbox{with $p_j(\alpha) \ne 0$}
$$
and $\sigma_j \ge 0$ { being} an integer. The sequence $0 \leq \sigma_1 = \sigma_1 \leq \sigma_2 = \sigma_2 \leq \cdots \leq \sigma_r = \sigma_r$ is called {\it the sequence of partial multiplicities} of $P(\lambda)$ at $\alpha$. The number $\alpha$ is a finite eigenvalue of $P(\lambda)$ if the partial multiplicity sequence at $\alpha$ contains at least two nonzero terms, or, equivalently, if $\alpha$ is a root of at least one invariant polynomial $g_j (\lambda)$. The {\it elementary divisors} of $P(\lambda)$ associated with a finite eigenvalue $\alpha$ is the collection of factors $(\lambda - \alpha)^{\sigma_1},(\lambda - \alpha)^{\sigma_1}, (\lambda - \alpha)^{\sigma_2},(\lambda - \alpha)^{\sigma_2}, \ldots , (\lambda - \alpha)^{\sigma_r}, (\lambda - \alpha)^{\sigma_r}$ for which $\sigma_j > 0$.

Let $P(\lambda) \in \POL^{ss}_{d,m\times m}$. Then, the sequence of partial multiplicities of $P(\lambda)$ at infinity is defined to be the sequence of partial multiplicities of the matrix polynomial $\rev_d P(\lambda):= \lambda^d P(1/\lambda)$ at zero. Moreover, if zero is an eigenvalue of $\rev_d  P(\lambda)$, then we say that $\lambda = \infty$ is an eigenvalue of the matrix polynomial $P(\lambda) \in \POL^{ss}_{d,m\times m}$. The elementary divisors for the zero eigenvalue of $\rev_d P(\lambda)$ are the elementary divisors associated with the infinite eigenvalue of $P(\lambda)$. We emphasize that the sequence of partial multiplicities of $P(\lambda)$ at infinity, as well as the fact that $\lambda = \infty$ is an eigenvalue of $P(\lambda)$, depend on the grade $d$ chosen for $P(\lambda)$. This observation plays a key role in the results of this paper.

Theorem \ref{gammath} connects the smallest partial multiplicity at $\infty$ with the grade and the degree of the matrix polynomial and with the fact of the leading coefficient being equal to zero. It is a consequence of \cite[Lemma 2.6]{ADHM23}. The complete proof can be found in \cite[Lemma 2.7]{ADHMarXiv}. Since Theorem \ref{gammath} is important in this paper we sketch the simple proof for the sake of completeness.
\begin{theorem}
\label{gammath}
Let $P(\lambda) = \sum_{i=0}^d \lambda^i P_i \in \POL_{d,m \times m}^{ss}$ with $\rank P=2r, \ r>0,$ and sequence of partial multiplicities at $\infty$ equal to $0 \leq \gamma_1=\gamma_1 \leq \gamma_2 = \gamma_2 \leq \dots \leq \gamma_r = \gamma_r$. Then, 
\begin{itemize}
    \item[\rm (a)] $\gamma_1 = d - \deg(P)$, and
    \item[\rm (b)] $P_d=0$ if and only if $1 \leq \gamma_1=\gamma_1 \leq \gamma_2 = \gamma_2 \leq \dots \leq \gamma_r = \gamma_r$.
\end{itemize}
\end{theorem}
\begin{proof} Observe that $P_d = 0$ if and only if $d > \deg(P)$. Therefore part (b) follows from part (a), so we only need to prove part (a). Note that $\rev_d P(\lambda) = \lambda^{d-\deg (P)} \rev_{\deg(P)} P(\lambda)$. Therefore, the (skew-symmetric) Smith form of $\rev_d P(\lambda)$ is the one of $\rev_{\deg(P)} P(\lambda)$ multiplied by $\lambda^{d-\deg (P)}$, which implies that { the smallest partial multiplicity of $P$ at $\infty$ is} $\gamma_1 = d-\deg (P) + \widetilde{\gamma}_1$, where $\widetilde{\gamma}_1$ is the smallest partial multiplicity at $0$ of $\rev_{\deg(P)} P(\lambda)$. The fact that $(\rev_{\deg(P)} P) (0) = P_{\deg(P)} \ne 0$, implies that $\widetilde{\gamma}_1 = 0$ and the result is proved.
\end{proof}

For an $m\times n$ matrix polynomial $P(\lambda)$, the left and right null-spaces, over the field of rational functions $\mathbb C(\lambda)$, are defined as follows:
\begin{align*}
{\cal N}_{\rm left}(P)&:= \{y(\lambda)^T \in \mathbb C(\lambda)^{1 \times m}: y(\lambda)^TP(\lambda) = 0_{1\times n} \}, \\
{\cal N}_{\rm right}(P)&:= \{x(\lambda) \in \mathbb C(\lambda)^{n\times 1}: P(\lambda)x(\lambda) = 0_{m\times 1}\}.
\end{align*}
Each subspace ${\cal V}$ of $\mathbb C(\lambda)^n$ has bases consisting entirely of vector polynomials.
A basis of ${\cal V}$ consisting of vector polynomials whose sum of degrees is minimal among all bases of ${\cal V}$ consisting of vector polynomials is called a {\it minimal basis} of ${\cal V}$. The ordered list of degrees of the vector polynomials in any minimal basis of ${\cal V}$ is always the same. These degrees are called the minimal indices of ${\cal V}$ \cite{Forn75,Kail80}. { This allows us to define the left and right {\it minimal indices} of a matrix polynomial $P(\lambda)$ \label{mi} as those of ${\cal N}_{\rm left}(P)$ and ${\cal N}_{\rm right}(P)$, respectively. Note that for a skew-symmetric matrix polynomial the left minimal indices are equal to the right ones.}

We define the {\it complete eigenstructure} of a skew-symmetric matrix polynomial { $P(\lambda) \in \POL^{ss}_{d, m\times m}$ {\it of grade $d$}} to be the collection of all finite and infinite eigenvalues, the corresponding elementary divisors { (or, equivalently, the corresponding sequences of partial multiplicities)}, and the left and right minimal indices of $P(\lambda)$. To refer concisely to the { subset of $\POL^{ss}_{d, m\times m}$ of skew-symmetric matrix polynomials with the same size $m\times m$, with the same grade $d$,} and with the same complete eigenstructure, we define the notion of orbit of a skew-symmetric matrix polynomial. Notably, an analogous definition for general polynomials is given in \cite{DmDo17}. { 
\begin{definition} \label{orbdef} Let $P(\lambda) \in \POL^{ss}_{d, m\times m}$. The subset of matrix polynomials in $\POL^{ss}_{d, m\times m}$ with the same complete eigenstructure as $P$ is called the \emph{orbit} of $P$, denoted $\orb(P)$.
\end{definition}
Observe that Theorem \ref{gammath}-(a) implies that all the polynomials in $\orb (P)$ have the same {\it degree}. Moreover, they also have the same rank, since the rank is determined by the size $m\times m$ and the number of left (or right) minimal indices}. Note that, by contrast to congruence orbits of skew-symmetric matrix pencils, Definition~\ref{orbdef} is not associated with an action of any group.

{
\subsection{Orbits of the same skew-symmetric matrix polynomial viewed with different grades} \label{subsec.cdd1} 

The proofs of the main results in this paper rely on comparing the properties of a skew-symmetric matrix polynomial $P(\lambda) \in \POL^{ss}_{d, m \times m}$ {\it of grade $d$} with the entry-wise identical polynomial $\widetilde{P} (\lambda) = \lambda^{d+1} \, 0 + P(\lambda) \in \POL^{ss}_{d+1, m \times m}$ {\it of grade $d+1$}, as well as on comparing the orbits, and their closures, of these two matrix polynomials. Lemma \ref{pori} establishes these comparisons. Recall in the statement of this lemma that $\orb(P) \subseteq \overline{\orb} (P) \subseteq \POL^{ss}_{d, m \times m}$ and that
$\orb(\widetilde{P}) \subseteq \overline{\orb} (\widetilde{P}) \subseteq \POL^{ss}_{d+1, m \times m}$.
\begin{lemma} \label{pori} Let $P(\lambda) \in \POL^{ss}_{d, m \times m}$ have rank $2r, \ r>0,$ and let $\widetilde{P} (\lambda) := \lambda^{d+1} \, 0 + P(\lambda) \in \POL^{ss}_{d+1, m \times m}$. Then:
\begin{itemize}
    \item[\rm (a)] $\rank P = \rank \widetilde{P} = 2r$.
    \item[\rm (b)] $P(\lambda)$ and $\widetilde{P} (\lambda)$ have the same finite eigenvalues, with the same sequences of partial multiplicities, or, equivalently, with the same elementary divisors.
    \item[\rm (c)] $P(\lambda)$ and $\widetilde{P} (\lambda)$ have the same left and right minimal indices.
    \item[\rm (d)]  $\gamma_1=\gamma_1 \leq \gamma_2 = \gamma_2 \leq \dots \leq \gamma_r = \gamma_r$ is the sequence of partial multiplicities at $\infty$ of $P(\lambda)$ if and only if
    $\gamma_1 +1=\gamma_1 +1\leq \gamma_2 +1= \gamma_2 +1 \leq \dots \leq \gamma_r +1 = \gamma_r +1$ is the sequence of partial multiplicities at $\infty$ of $\widetilde{P}(\lambda)$.
    \item[\rm (e)] If $\widetilde{Q} (\lambda) = \sum_{i=0}^{d+1} \lambda^i Q_i \in \orb (\widetilde{P})$, then $Q_{d+1} = 0$.
     \item[\rm (f)] If $\widetilde{Q} (\lambda) = \sum_{i=0}^{d+1} \lambda^i Q_i \in \overline{\orb} (\widetilde{P})$, then $Q_{d+1} = 0$.
     \item[\rm (g)] Let $\widetilde{Q}(\lambda) = \lambda^{d+1} \, 0 + Q(\lambda) \in \POL^{ss}_{d+1, m \times m}$, where $Q(\lambda) \in \POL^{ss}_{d, m \times m}$. Then,
     \begin{enumerate}
     \item $\widetilde{Q} \in \orb (\widetilde{P})$ if and only if $Q \in \orb (P)$, 
     \item $\widetilde{Q} \in \overline{\orb} (\widetilde{P})$ if and only if $Q \in \overline{\orb} (P)$, and 
     \item $\orb(\widetilde{Q}) \subseteq \overline{\orb} (\widetilde{P})$ if and only if $\orb (Q) \subseteq \overline{\orb} (P)$.
     \end{enumerate}
\end{itemize}
\end{lemma}
\begin{proof}
Parts (a), (b) and (c) are trivial because the rank, the (skew-symmetric) Smith form and the left and right rational null spaces of a matrix polynomial do not depend on the grade chosen for that polynomial.

Part (d) follows from the fact that $\rev_{d+1} \widetilde{P} (\lambda) = \lambda \rev_{d} P (\lambda)$. Therefore, the (skew-symmetric) Smith form of $\rev_{d+1} \widetilde{P} (\lambda)$ is equal to $\lambda$ times the Smith form of $\rev_{d} P (\lambda)$, which implies the result. 

Part (e). Since all the matrix polynomials in $\orb(\widetilde{P}) \subseteq \POL^{ss}_{d+1, m \times m}$ have the same complete eigenstructure, all of them have the same sequence of partial multiplicities at $\infty$. Then, Theorem \ref{gammath}-(a) implies that all of them have the same degree. So, $Q_{d+1} = 0$ because the coefficient of $\lambda^{d+1}$ in $\widetilde{P}$ is $P_{d+1} =0$.

Part (f).  $\widetilde{Q} (\lambda) = \sum_{i=0}^{d+1} \lambda^i Q_i \in \overline{\orb} (\widetilde{P})$ if and only if  $\widetilde{Q}$ is the limit of a sequence of matrix polynomials in $\orb (\widetilde{P})$. But, by (e), all the terms in this sequence have the coefficient of $\lambda^{d+1}$ equal to zero. So, the limit of the sequence have also the coefficient of $\lambda^{d+1}$ equal to zero.

Part (g)-1 follows from the definition of orbit and from applying (b), (c), and (d) to $P$ and $\widetilde{P}$ and to $Q$ and $\widetilde{Q}$.

Part (g)-2. If $\widetilde{Q} \in \overline{\orb} (\widetilde{P})$, then $\widetilde{Q} = \lim_{k\rightarrow \infty} \widetilde{Q}^{(k)}$, for a sequence $\{\widetilde{Q}^{(k)}\}_{k \in \mathbb{N}}$ whose terms belong to $\orb (\widetilde{P}) \subset \POL^{ss}_{d+1, m \times m}$. Taking into account (e), we have that 
$\widetilde{Q}^{(k)}(\lambda) = \lambda^{d+1} \, 0 + Q^{(k)}(\lambda)$, where $\{Q^{(k)}\}_{k \in \mathbb{N}} \subset \POL^{ss}_{d, m \times m}$, and, taking into account (g)-1, $\{Q^{(k)}\}_{k \in \mathbb{N}} \subset \orb (P)$. Thus, $Q = \lim_{k\rightarrow \infty} Q^{(k)}$, which implies that $Q \in \overline{\orb} (P)$.

Conversely, if $Q \in \overline{\orb} (P)$, then  $Q = \lim_{k\rightarrow \infty} Q^{(k)}$, for a sequence $\{Q^{(k)}\}_{k \in \mathbb{N}}$ whose terms belong to $\orb (P) \subset \POL^{ss}_{d, m \times m}$. By part (g)-1, the sequence $\{\lambda^{d+1} \, 0 + Q^{(k)}\}_{k \in \mathbb{N}}$ is included in $\orb (\widetilde{P}) \subset \POL^{ss}_{d+1, m \times m}$. Moreover, $\lim_{k \rightarrow \infty} \lambda^{d+1} \, 0 + Q^{(k)} = \widetilde{Q}$, which implies that $\widetilde{Q} \in \overline{\orb} (\widetilde{P})$.

{ Finally, for part (g)-3, let us first assume that $\orb(\widetilde Q)\subseteq\overline\orb(\widetilde P).$ Let $Q'\in\orb(Q)$. We want to prove that $Q'\in\overline \orb (P)$. Since $Q'\in\orb(Q)$, by part (g)-1 we conclude that $\widetilde Q':=\lambda^{d+1}0+Q'\in\orb(\widetilde Q)$, and then $\widetilde Q'\in\overline\orb(\widetilde P)$, by hypothesis. But now part (g)-2 implies $Q'\in\overline\orb(P)$.

Conversely, let us assume that $\orb(Q)\subseteq\overline\orb(P)$, and let $\widetilde Q'\in\orb(\widetilde Q)$. We want to prove that $\widetilde Q'\in\overline\orb(\widetilde P)$. By part (e) we know that $\widetilde Q'=\lambda^{d+1}0+Q'$, for some $Q'\in{\rm POL}^{ss}_{d,m\times m}$, so part (g)-1 together with the hypothesis $\widetilde Q'\in\orb(\widetilde Q)$ imply that $Q'\in\orb(Q)$. But this in turn implies $Q'\in\overline\orb(P)$, by hypothesis, so  part (g)-2 gives $\widetilde Q'\in\overline\orb(\widetilde P)$, as wanted.
}
\end{proof}
}

\subsection{Linearization of skew-symmetric matrix polynomials and their perturbations}

A matrix pencil ${\cal F}_{P}$ is called a {\it linearization} of a matrix polynomial $P(\lambda)$ { of grade $d$ if ${\cal F}_P$ has the same finite eigenvalues and associated elementary divisors}, the same number of left minimal indices, and the same number of right minimal indices as $P$ \cite[Thm. 4.1]{DeDM14}. If in addition, { $\rev_1 {\cal F}_{P}$ is a linearization of $\rev_d P$} then ${\cal F}_{P}$ is called a {\it strong linearization} of $P$ and, then, ${\cal F}_{P}$ and $P$ have also the same infinite elementary divisors.

The following pencil-template is known to be a skew-symmetric strong linearization of skew-symmetric  matrix polynomials $P(\lambda) = \sum_{i=0}^d \lambda^i A_i \in \POL^{ss}_{d, m\times m}$ of odd grade $d$ \cite{MMMM13}, see also \cite{AnVo04,MMMM10}:
\begin{align*}
{   {\cal F}}_{P}(i,i)&=\begin{cases}
\lambda A_{d-i+1} + A_{d-i} & \text{if } i \text{ is odd,}\\
0 & \text{if } i \text{ is even,}\\
\end{cases}\\
{   {\cal F}}_{P}(i,i+1)&=\begin{cases}
- I_m & \text{if } i \text{ is odd,}\\
- \lambda I_m & \text{if } i \text{ is even,}\\
\end{cases} \quad
{   {\cal F}}_{P}(i+1,i)=\begin{cases}
I_m & \text{if } i \text{ is odd,}\\
\lambda I_m & \text{if } i \text{ is even,}\\
\end{cases}
\end{align*}
where ${\cal F}_{P}(j,k)$ denotes an $m \times m$ matrix pencil which is at the position $(j,k)$ of the block pencil ${\cal F}_{P}$ and  $j,k=1, \dots , d$. The blocks of ${\cal F}_{P}$ in positions which are not specified above are zero. We rewrite this strong linearization template in  a matrix form as follows:
\begin{equation}
\label{linform}
{\cal F}_{P}(\lambda)=
\lambda
\begin{bmatrix}
A_d&&&&&\\
&\ddots&\ddots&&&\\
&\ddots&0&-I&&\\
&&I&A_{3}&&\\
&&&&0&-I\\
&&&&I&A_1\\
\end{bmatrix} - \begin{bmatrix}
-A_{d-1}&I&&&&\\
-I&0&\ddots&&&\\
&\ddots&\ddots&&&\\
&&&-A_{2}&I&\\
&&&-I&0&\\
&&&&&-A_0\\
\end{bmatrix}.
\end{equation}
Note that the linearization ${   {\cal F}}_{P}$ is defined { only} for skew-symmetric matrix polynomials of odd grade. The reason is that there is no skew-symmetric linearization-template (i.e., a skew-symmetric companion form in the language of \cite[Sects. 5 and 7]{DeDM14}) for skew-symmetric matrix polynomials of even grade \cite{DeDM14,MMMM13}.

{ For any skew-symmetric matrix polynomial $P(\lambda) \in \POL^{ss}_{d, m\times m}$, the strong linearization \eqref{linform}} preserves the finite and infinite elementary divisors of $P(\lambda)$ but does not preserve the left and right minimal indices of $P(\lambda)$. Nevertheless, the relations between the minimal indices of a skew-symmetric matrix polynomial $P(\lambda)$ and its linearization \eqref{linform} are derived in \cite{Dmyt15}, see also \cite{DeDM10}. We recall them in Theorem \ref{skewlin}.

\begin{theorem}{\rm \cite{Dmyt15}} \label{skewlin}
Let $P(\lambda)$ be a skew-symmetric $m\times m$ matrix polynomial of odd grade $d \ge 3$, and let ${\cal F}_{P}$ be its { strong} linearization \eqref{linform} given above. If $0 \le  \varepsilon_1 \le \varepsilon_2 \le~\dots \le \varepsilon_t$ are the right (left) minimal indices of $P$ then
$$0 \le  \varepsilon_1 + \frac{1}{2}(d-1) \le \varepsilon_2 + \frac{1}{2}(d-1) \le  \cdots \le \varepsilon_t + \frac{1}{2}(d-1)$$
are the right (left) minimal indices of ${   {\cal F}}_{P}$.
\end{theorem}

The linearization ${\cal F}_{P}$ in \eqref{linform} is crucial for obtaining the results in Section~\ref{sec.main}. Thus we define the {\it generalized Sylvester space} consisting of the linearizations ${\cal F}_{P}$ of all $m \times m$ skew-symmetric matrix polynomials of odd grade $d$, namely:
\[ { 
\GSYL^{ss}_{d,m\times m}= \{ {\cal F}_{P}  \ : P(\lambda) \in \POL^{ss}_{d, m\times m} \; \text{with odd} \; d \}.}
\]
If there is no risk of confusion we will write $\GSYL$ instead of $\GSYL^{ss}_{d, m\times m}$. {   Given ${\cal F}_{P} = \lambda A - B$ and ${\cal F}_{P'} = \lambda A' - B'$, the function $d({\cal F}_{P},{\cal F}_{P'}):= \left( ||A-A'||_F^2 + ||B-B'||_F^2 \right)^{\frac{1}{2}}$} mentioned in Sections \ref{sec.main1} and \ref{sect.prempolys} is a distance on $\GSYL$ and it makes $\GSYL$ a metric space.  Since $d({   {\cal F}}_{P}, {   {\cal F}}_{P'}) = d(P,P')$, there is a bijective isometry (and therefore a homeomorphism):
$$
f: \POL^{ss}_{d, m\times m} \rightarrow \GSYL^{ss}_{d, m\times m} \quad \text{such that} \quad {   f(P) = {\cal F}_{P}}.
$$
We also define the {\em orbit} of the skew-symmetric linearizations of type \eqref{linform} of a fixed skew-symmetric matrix polynomial { $P\in \POL^{ss}_{d, m\times m}$ of odd grade $d$ as}
\begin{equation}\label{linorb}
\orb({\cal F}_{P}) = \{(S^{T}   {\cal F}_{P} S) \in \GSYL^{ss}_{d, m\times m} \ : \ S \in  GL_{md}(\mathbb C) \}.
\end{equation}
We emphasize that all the elements of $\orb({   {\cal F}}_{P})$ have the block structure of the elements of $\GSYL^{ss}_{d, m\times m}$.
Thus, in particular, $\orb(P) = f^{-1}(\orb({   {\cal F}}_{P}))$, as a consequence of the properties of strong linearizations and Theorem \ref{skewlin}, and $\overline{\orb}(P) = f^{-1}(\overline{\orb}({   {\cal F}}_{P}))$, as a consequence of $f$ being a homeomorphism. Moreover, we also have that, for any $m \times m$ skew-symmetric matrix polynomials $P, Q$ of odd grade $d$, $\overline{\orb}(P) \supseteq \overline{\orb}(Q)$ if and only if $\overline{\orb}({   {\cal F}}_{P}) \supseteq \overline{\orb}({   {\cal F}}_{Q})$, where it is essential to note that the closures are taken in the metric spaces $\POL$ and $\GSYL$, respectively, defined above.  Similarly to the unstructured matrix pencil case, $\orb({{\cal F}}_{P})$ is open in its closure in the relative Euclidean topology, and so is $\orb(P)$ since $f$ is a homeomorphism.

\section{Generic skew-symmetric matrix polynomials with bounded rank and fixed grade}
\label{sec.main}

In this section we present the complete eigenstructure of the generic $m\times m$ skew-symmetric matrix polynomials of rank at most $2r$ and even grade $d$ in Theorem \ref{mainth}. Since the case of odd grade is known, we also provide a result that does not depend on the parity of $d$ { in Theorem \ref{anydth}}. We start by proving the following auxiliary lemma.

{
\begin{lemma}\label{srank}
Let $d, m$, and $r$ be integers such that $d\geq 1$, $m\geq 2$, and $2 \leq 2 r \leq (m-1)$. For any skew-symmetric matrix polynomial $Q (\lambda) \in \POL^{ss}_{d, m \times m}$, with $\rank Q=2r_1$ and $r_1<r$, there exists a sequence of skew-symmetric matrix polynomials $\{ P^{(k)} (\lambda)\}_{k \in \mathbb{N}} \subset \POL^{ss}_{d, m \times m}$ with $\rank P^{(k)} = 2r$, for all $k \in \mathbb{N}$, such that $\lim_{k \rightarrow \infty} P^{(k)}  = Q$.
\end{lemma}
\begin{proof} Let $\mu \in \mathbb{C}$ be such that $\rank Q(\mu) = 2 r_1$, that is, $\mu$ is not an eigenvalue of $Q(\lambda)$. Since the constant matrix $Q(\mu) \in \mathbb{C}^{m\times m}$ is skew symmetric, there exist a unitary matrix $U \in \mathbb{C}^{m\times m}$ and positive real numbers $s_1, \ldots , s_{r_1}$ such that (see \cite[Corollary 2.6.6]{HoJo2013})
\begin{equation*}
Q(\mu)  = U \left(
\begin{bmatrix}
0 & s_1\\
-s_1 & 0
\end{bmatrix} \oplus
\dots
\oplus
\begin{bmatrix}
0 & s_{r_1}\\
-s_{r_1} & 0
\end{bmatrix} \oplus
0_{m-2r_1} \right) U^T .
\end{equation*}
Then, we define the following constant $m\times m$ skew-symmetric matrix
\begin{equation*}
E :=
U \left(
0_{2r_1} \oplus
\begin{bmatrix}
0 & 1\\
-1 & 0
\end{bmatrix} \oplus
\dots
\oplus
\begin{bmatrix}
0 & 1\\
-1 & 0
\end{bmatrix} \oplus
0_{m-2r}\right) U^T,
\end{equation*}
where there are $r-r_1$ blocks equal to $\left[\begin{smallmatrix}
0 & 1\\
-1 & 0
\end{smallmatrix}\right]$. Then, the desired sequence of skew-symmetric matrix polynomials is defined as $P^{(k)}(\lambda):=Q(\lambda)+\frac{1}{k} E$. Note that $\rank P^{(k)} =2r$, because $\rank P^{(k)}\leq \rank Q + \rank (\frac{1}{k} E) = 2 r_1 + 2 r - 2 r_1 = 2 r$ and $\rank P^{(k)} \geq \rank P^{(k)}(\mu) = \rank (Q(\mu) + \frac{1}{k} E) = 2 r$.
\end{proof}
}

We recall the following lemma that reveals,  for  skew-symmetric matrix polynomials $P$ with  odd  grade $d$, a relation between $\overline{\orb}({\cal F}_{P})$, where the closure is taken in $\GSYL^{ss}_{d,m\times m}$, and $\overline{\orb^c}({   {\cal F}}_{P})$, where the closure is taken in $\PEN^{ss}_{n \times n}$, with $n = m d$.

\begin{lemma} {\rm\cite[Lemma 5.1]{DmDo18}}  \label{cl}
Let $P$ be an $m\times m$ skew-symmetric matrix polynomial with odd grade $d$ and ${   {\cal F}}_{P}$ be its linearization  \eqref{linform}. Then $\overline{\orb}({   {\cal F}}_{P})= \overline{\orb^c}({   {\cal F}}_{P}) \cap \GSYL^{ss}_{d,m\times m}$.
\end{lemma}

{ We are finally in the position of stating and proving the main result of this paper.}
\begin{theorem} \label{mainth}
Let $d,m$, and $r$ be integers such that $d \geq 1$ is even, $m \geq 2$, and $2 \leq 2r \leq (m-1)$.
The set of $m\times m$ complex skew-symmetric matrix polynomials of grade $d$ with rank at most $2r$ is a closed subset of $\POL_{d,m\times m}^{ss}$ equal to $\overline{\orb}(W)$, where { $W \in \POL^{ss}_{d, m\times m}$} is an $m \times m$ complex skew-symmetric matrix polynomial of degree exactly $d$ and rank exactly $2r$ with no elementary divisors at all, with $t$ left minimal indices equal to $(\beta +1)$ and with $(m-2r-t)$ left minimal indices equal to $\beta$, where $\beta = \lfloor rd / (m-2r) \rfloor$ and $t \equiv rd \mod (m-2r)$, and with the right minimal indices equal to the left minimal indices.
\end{theorem}
\begin{proof} 
By Lemma \ref{srank}, each skew-symmetric matrix polynomial in $\POL^{ss}_{d, m\times m}$ of rank $2r_1$, with $r_1 < r$, is in the closure of the subset of $\POL^{ss}_{d, m\times m}$ formed by the matrix polynomials of rank exactly $2r$. Therefore it remains to show that any skew-symmetric matrix polynomial of rank exactly $2r$ in $\POL^{ss}_{d, m\times m}$ is in the closure of the orbit of $W$  as in the statement.

Denote the complete eigenstructure described in the statement (consisting only of the left and right minimal indices) by
\begin{equation}
\label{kcilist}
{\bf W}: \bigg\{\overbrace{\underbrace{\beta+1, \dots , \beta+1}_{t},\underbrace{\beta, \dots , \beta}_{m-2r-t}}^{\text{left minimal indices}}, \overbrace{\underbrace{\beta+1, \dots , \beta+1}_{t}, \underbrace{\beta, \dots , \beta}_{m-2r-t}}^{\text{right minimal indices}} \bigg\}.
\end{equation}
First we show that there exists an $m\times m$ skew-symmetric matrix polynomial { $W \in \POL^{ss}_{d, m\times m}$} of degree exactly $d$ and rank exactly $2r$ that has the complete eigenstructure~${\bf W}$ {   in \eqref{kcilist}}. By \cite[Thm. 3.3]{Dmyt15} it is enough to show that the sum of the left (or right) minimal indices of ${\bf W}$ is equal to $rd$:
\begin{align*}
&\sum_1^t (\beta +1) + \sum_1^{m-2r-t} \beta
= \sum_1^{m-2r} \beta + t = (m-2r) \lfloor rd/(m-2r) \rfloor + t = rd.
\end{align*}

{ Next, as in Subsection \ref{subsec.cdd1}, we add the term $\lambda^{d+1} \, 0$ to $W$ and denote the (entry-wise identical) result by $\widetilde{W}$, i.e., $\widetilde{W} = \lambda^{d+1} \, 0 + W \in \POL^{ss}_{d+1, m\times m}$, in order to view the polynomial as a polynomial of odd grade $d+1$. Parts (a)--(d) in Lemma \ref{pori} imply that the complete eigenstructure of $\widetilde{W}$ is given by the minimal indices in \eqref{kcilist} together with $2r$ infinite elementary divisors with degree equal to one, i.e., the sequence of partial multiplicities at $\infty$ of $\widetilde{W}$ is $\gamma_1=\gamma_1 = \gamma_2=\gamma_2 = \dots = \gamma_{r} = \gamma_r =1$.}

For every { skew-symmetric} $m \times m$ matrix polynomial $P$ of odd grade $d+1$ and rank at most $2r$, the linearization ${\cal F}_{P}$ in \eqref{linform} has rank { equal to $\rank(P) + md$, i.e., rank at most $2r+md$}, because ${\cal F}_{P}$ is unimodularly equivalent to $P \oplus I_{md}$.
The linearization ${   {\cal F}}_{\widetilde{W}}$ of the matrix polynomial $\widetilde{W}$ is an $m(d+1) \times m(d+1)$ skew-symmetric matrix pencil with rank $2r+md$ and, by Theorem \ref{skewlin}, the KCF of ${   {\cal F}}_{\widetilde{W}}$ is the direct sum of the following blocks:
\begin{equation} \label{cka}
\begin{split}
\Big\{\underbrace{{\cal L}_{\beta+\eta+1}, \dots , {\cal L}_{\beta+\eta+1}}_{t},\underbrace{{\cal L}_{\beta+\eta}, \dots , {\cal L}_{\beta+\eta}}_{m-2r-t}, \underbrace{{\cal L}_{\beta+\eta+1}^T, \dots , {\cal L}_{\beta+\eta+1}^T}_{t}, \underbrace{{\cal L}_{\beta+\eta}^T, \dots , {\cal L}_{\beta+\eta}^T}_{m-2r-t}, \\
\underbrace{{\cal E}_{1}(\infty),\dots,{\cal E}_{1}(\infty)}_{2r}\Big\},
%\underbrace{{\cal K}_{1}, \dots , {\cal K}_{1}}_{r}\},
\end{split}
\end{equation}
where $\eta=\frac{1}{2}d$.  Next, we show that the KCF of ${\cal F}_{\widetilde{W}}$ as in \eqref{cka} coincides with the KCF of the most generic skew-symmetric matrix pencil ${\cal W}$ of rank $2w=2r+md$ and size $n \times n$, where $n = m (d+1)$, having exactly $r$ { ${\cal K}$-blocks} associated to the infinite eigenvalue in its skew-symmetric KCF. This generic KCF is given in Theorem~\ref{th:improvedDeDo}, namely:
\begin{equation} \label{pg}
\begin{split}
\Big\{\underbrace{{\cal L}_{\alpha+1}, \dots , {\cal L}_{\alpha+1}}_{s},\underbrace{{\cal L}_{\alpha}, \dots , {\cal L}_{\alpha}}_{n-2w-s}, \underbrace{{\cal L}_{\alpha+1}^T, \dots , {\cal L}_{\alpha+1}^T}_{s}, \underbrace{{\cal L}_{\alpha}^T, \dots , {\cal L}_{\alpha}^T}_{n-2w-s},
\underbrace{{\cal E}_{1}(\infty),\dots,{\cal E}_{1}(\infty)}_{2r}\Big\}.
\end{split}
\end{equation}
In other words, we  are going to show that the numbers and the sizes of the ${\cal L}$ and ${\cal L}^T$ blocks in \eqref{cka} and \eqref{pg} coincide, namely $\beta  + d/2 = \alpha$, $s=t$, and $m-2r-t = n - 2w - s$.

For the sizes of the blocks we have
\begin{align*}
\beta  +  \frac{d}{2}& =\left\lfloor \frac{rd}{m-2r} \right\rfloor + \frac{d}{2} =  \left\lfloor  \frac{(m-2r)d + 2rd}{2(m-2r)} \right\rfloor  =  \left\lfloor  \frac{ md}{2(m-2r)} \right\rfloor  \\
& = \left\lfloor  \frac{md+2r - 2r}{2(m(d+1) - (md + 2r))} \right\rfloor = \left\lfloor  \frac{2w - 2r}{2(n - 2w)} \right\rfloor = \alpha.
\end{align*}

For the numbers of the blocks, we have
\begin{align*}
t &\equiv rd  \equiv(m-2r)d/2 + rd \mod (m-2r) \\
&\equiv (md+2r - 2r)/2 \mod (m(d+1) - (md + 2r))\\
&\equiv w-r \equiv s\mod (n - 2w)
\end{align*}
and, since $s,t < n-2w$, it must be $t=s$. Moreover,
$$m-2r-t = m(d+1) - md - 2r -t= n - 2w - s.$$

Thus ${{\cal F}}_{\widetilde{W}}$ is congruent to the most generic skew-symmetric matrix pencil $\mathcal{W}$, with exactly $r$ { ${\cal K}$-blocks corresponding to the infinite eigenvalue in its skew-symmetric KCF and of rank $2w = 2r + md$}, obtained in Theorem \ref{th:improvedDeDo}, since they both are skew-symmetric and have the same KCF. Therefore $\orb^c({   {\cal F}}_{\widetilde{W}})=\orb^c(\mathcal{W})$.

{ Given any $P \in \POL^{ss}_{d,m\times m}$ of rank exactly $2r$, the entry-wise identical matrix polynomial $\widetilde{P} = \lambda^{d+1} \, 0 + P \in \POL^{ss}_{d+1,m\times m}$ has rank exactly $2r$ and exactly $2r$ infinite elementary divisors, by Lemma \ref{pori}-(a)-(d). So, by Theorem \ref{th:improvedDeDo}}, we have that $\overline{\orb^c}(\mathcal{W})  \supseteq \orb^c ({   {\cal F}}_{\widetilde{P}})$, thus $\overline{\orb^c}({   {\cal F}}_{\widetilde{W}}){ =\overline{\orb^c}({\cal W})} \supseteq \orb^c ({   {\cal F}}_{\widetilde{P}})$. Therefore { $\overline{\orb^c}({   {\cal F}}_{\widetilde{W}}) \cap \GSYL^{ss}_{d+1,m\times m}  \supseteq \orb^c({   {\cal F}}_{\widetilde{P}}) \cap \GSYL^{ss}_{d+1,m\times m}$}, and Lemma~\ref{cl} implies
$\overline{\orb}({   {\cal F}}_{\widetilde{W}})  \supseteq \orb({   {\cal F}}_{\widetilde{P}})$. According to the discussion after \eqref{linorb}, $\overline{\orb}({   {\cal F}}_{\widetilde{W}}) \supseteq \orb({   {\cal F}}_{\widetilde{P}})$ is equivalent to $\overline{\orb}(\widetilde{W}) \supseteq \orb(\widetilde{P})$ and by { Lemma~\ref{pori}-(g)-3} the latter is equivalent to $\overline{\orb}(W) \supseteq \orb(P)$. Therefore any $m\times m$ skew-symmetric matrix polynomial of grade $d$, with rank exactly $2r$ { belongs to} $\overline{\orb}(W)$. This completes the proof (we remind that the skew-symmetric matrix polynomials of grade $d$ with rank smaller than $2r$ are treated at the beginning of this proof).
\end{proof}

\begin{example}
In Table {\rm\ref{table2}} we provide examples of the eigenstructures (namely, the left minimal indices) of generic skew-symmetric matrix polynomials of grade $2$ with various sizes (from $3 \times 3$ to $7\times 7$) and ranks ($2,4$, and $6$) obtained from Theorem {\rm\ref{mainth}}.

\begin{table}[hbt!] 
\centering
\begin{tabular}{|r|c|c|c|c|c|}
  \hline
         & $3\times 3$& $4\times 4$   & $5\times 5$   & $6\times 6$   &$7\times 7$ \\
  \hline
  rank 2 & \{2\}  & \{1, 1\}  & \{1, 1, 0\} & \{1, 1, 0, 0\} & \{1, 1, 0, 0, 0\} \\
  rank 4 & - & - & \{4\}  & \{2, 2\}  & \{2, 1, 1\}  \\
  rank 6 & - & - & -  & - & \{6\}  \\
  \hline
\end{tabular}
  \caption{\label{table2} Left minimal indices of the generic skew-symmetric matrix polynomial of grade~2 and of certain sizes (columns correspond to different sizes) and rank (rows correspond to different ranks). Note that the right minimal indices are equal to the left minimal indices, and no elementary divisors are present in these polynomials.}
  \end{table}
\end{example}

Note that Theorem \ref{mainth} is an even-grade-version of the corresponding theorem for the odd grade { skew-symmetric} matrix polynomials \cite[Thm. 5.2]{DmDo18}. This allows us to state the following theorem that does not depend on the parity of $d$.

\begin{theorem} \label{anydth}
Let $m,r$, and $d$ be integers such that $m \geq 2$, $d \geq 1$, and $2 \leq 2r \leq (m-1)$.
The set of $m\times m$ complex skew-symmetric matrix polynomials of grade $d$ with rank at most $2r$ is a closed subset of $\POL_{d,m\times m}^{ss}$ equal to $\overline{\orb}(W)$, where { $W \in \POL^{ss}_{d, m\times m}$} is an $m \times m$ complex skew-symmetric matrix polynomial of degree exactly $d$ and rank exactly $2r$ with no elementary divisors at all, with $t$ left minimal indices equal to $(\beta +1)$ and with $(m-2r-t)$ left minimal indices equal to $\beta$, where $\beta = \lfloor rd / (m-2r) \rfloor$ and $t = rd \mod (m-2r)$, and with the right minimal indices equal to the left minimal indices.
\end{theorem}
\begin{proof}
If $d$ is odd then the statement coincides with \cite[Thm. 5.2]{DmDo18}, and if $d$ is even then it coincides with Theorem \ref{mainth}.
\end{proof}

\section{Codimension computations} \label{sec.codim}

We start by recalling that the congruence orbit of an $n\times n$ skew-symmetric pencil $\lambda A - B$ is a {   differential} manifold in the complex $n^2-n$ dimensional space of skew-symmetric pencils { $\PEN^{ss}_{n \times n}$}. For  $\lambda A - B$, define the {\it dimension} of $\orb^c (\lambda A - B)$
to be the dimension of the tangent space to this orbit at the point $\lambda A - B$, namely:
\begin{equation*}
\tsp_{\lambda A - B}^c:=\{\lambda (X^TA+AX) - (X^TB + BX):
X\in{\mathbb
C}^{n\times n}\}
\end{equation*}
The orthogonal complement (with respect to the Frobenius inner product) to $\tsp_{\lambda A - B}^c$,
is called the {\em normal space} to $\orb^c (\lambda A - B)$ at the point $\lambda A - B$. The dimension of the normal space is the {\it codimension} of the congruence orbit of $\lambda A - B$ and is equal to $n(n-1)$ minus the dimension of the congruence orbit of $\lambda A - B$.
Explicit expressions for the codimensions of congruence orbits of skew-symmetric pencils in $\PEN_{n\times n}^{ss}$ are presented in \cite{DmKS13}, see also \cite{Dmyt16}, and implemented in the MCS (Matrix Canonical Structure) Toolbox \cite{DmJK13}.
Recall that, for $Q \in \POL_{d, m\times m}^{ss}$ of odd grade, we define the codimension of $\orb (Q) \subset \POL_{d, m\times m}^{ss}$ to be
$ \cod_{\POL} \orb (Q) := \cod_{\GSYL} \orb({   {\cal F}}_{Q})$, where by $\cod_X (\orb)$ we mean the codimension of the orbit $\orb$ in the space $X$, see \cite{Dmyt15,DmDo18}.
Moreover,
$\cod_{\GSYL} \orb({{\cal F}}_{P})= \cod_{\PEN} \orb^{c}({{\cal F}}_{P})$, by \cite[Sect.~6]{Dmyt15}.

Now let $P\in \POL_{d, m\times m}^{ss}$ be a matrix polynomial of even grade $d$ and { $\widetilde{P} := 0 \, \lambda^{d+1} + P \in \POL_{d+1, m\times m}^{ss}$ be a matrix polynomial of odd grade $d+1$}. Denote by { $\GSYL_{d+1, m\times m}^{ss,0}$} the space of the linearizations \eqref{linform} of { skew-symmetric} polynomials $\widetilde{P}$ of odd grade $d+1$ with { zero coefficient for $\lambda^{d+1}$}. For a polynomial $P$ of even grade $d$, we define the codimension of the orbit, $\cod_{\POL} \orb (P)$, as $\cod_{\GSYL^{0}} \orb ({\cal F}_{\widetilde{P}})$, i.e., the codimension of $\orb ({\cal F}_{\widetilde{P}})$ in the space { $\GSYL_{d +1, m\times m}^{ss,0}$}. Recall also that the dimension of $\orb ({\cal F}_{\widetilde{P}})$ in the space { $\GSYL_{d+1, m\times m}^{ss}$} is defined as the dimension of { $\tsp_{{\cal F}_{\widetilde{P}}}^c \cap \GSYL_{d+1, m\times m}^{ss}$} and, similarly, the dimension of { $\tsp_{{\cal F}_{\widetilde{P}}}^c \cap \GSYL_{d+1, m\times m}^{ss,0}$} is equal to the dimension of $\orb ({\cal F}_{\widetilde{P}})$ in { $\GSYL_{d+1, m\times m}^{ss,0}$}. In Lemma \ref{dimeq} we show that these two dimensions are equal to each other.

\begin{lemma}
\label{dimeq}
Let $\widetilde{P}= 0 \, \lambda^{d+1} + P_d \lambda^{d} + \dots + P_1 \lambda + P_0 \, { \in \POL^{ss}_{d+1, m\times m}}$ be a { skew-symmetric} matrix polynomial. The dimension of $\orb({{\cal F}}_{\widetilde{P}})$ in { $\GSYL_{d+1, m\times m}^{ss}$}is equal to the dimension of $\orb({{\cal F}}_{\widetilde{P}})$ in { $\GSYL_{d+1, m\times m}^{ss,0}$}.
\end{lemma}
\begin{proof}
The tangent space $\tsp_{{\cal F}_{\widetilde{P}}}^c = { \{X^T {\cal F}_{\widetilde{P}} + {\cal F}_{\widetilde{P}} X : X\in{\mathbb C}^{n\times n}\}}$ to $\orb^c({{\cal F}}_{\widetilde{P}})$ at the point ${\cal F}_{\widetilde{P}}$ consist of the skew-symmetric pencils of the following shape:
\begin{equation}\label{tansp}
    { X^T {\cal F}_{\widetilde{P}} + {\cal F}_{\widetilde{P}} X } =
\lambda \begin{bmatrix}
    0_{m\times m}&*&\dots&*\\
    *&*&\dots&*\\
    \vdots&\vdots&&\vdots\\
    *&*&\dots&*\\
\end{bmatrix} -
\begin{bmatrix}
    *&*&\dots&*\\
    *&*&\dots&*\\
    \vdots&\vdots&&\vdots\\
    *&*&\dots&*\\
\end{bmatrix},
\end{equation}
where $*$ denote possibly non-zero entries. The zero block in the $(1,1)$ entry of the $\lambda$-matrix of \eqref{tansp} ensures that $\tsp_{{\cal F}_{\widetilde{P}}} \cap  \GSYL_{d+1, m\times m}^{ss}$ is equal to $\tsp_{{\cal F}_{\widetilde{P}}} \cap  \GSYL_{d+1, m\times m}^{ss,0}$, and so their dimensions are equal to each other.
\end{proof}

Recall that the dimension and codimension of any given orbit sum up to the dimension of the whole space. Note also that  $\dim \GSYL_{d+1, m\times m}^{ss} = \dim \GSYL_{d+1, m\times m}^{ss,0} + \frac{m(m-1)}{2}$. Thus, taking into account Lemma \ref{dimeq}, the codimension of $\orb({{\cal F}}_{\widetilde{P}})$ in $\GSYL_{d+1, m\times m}^{ss}$ has to be $\frac{m(m-1)}{2}$  larger than the codimension of $\orb({{\cal F}}_{\widetilde{P}})$ in $\GSYL_{d+1, m\times m}^{ss,0}$.

Therefore, for the generic $m\times m$ skew-symmetric matrix polynomial $P$ with even grade $d$ identified in Theorem \ref{mainth}, we have
\begin{align} \nonumber
\cod_{\POL} \orb(P) &= \cod_{\GSYL^0} \orb({{\cal F}}_{\widetilde{P}}) = \cod_{\GSYL} \orb({{\cal F}}_{\widetilde{P}}) - \frac{m(m-1)}{2} \\ \nonumber
&= \cod_{\PEN} \orb^{c} ({\cal F}_{\widetilde{P}}) - \frac{m(m-1)}{2}\\  \label{codequalities} & = \cod_{\PEN} \orb^{c} ({\cal W}) - \frac{m(m-1)}{2},
\end{align}
where ${\cal W}$ is the pencil in Theorem \ref{th:improvedDeDo} with the identifications $n=m(d+1)$ and $w = (md + 2r)/2$.
Using \cite[Thm.~3]{DmKS13}, see also \cite[Thm. 2.8]{DmJK13}, applied to the skew-symmetric pencil ${\cal W}$ in \eqref{max}, the codimension of $\orb^c ({\cal W})$ in $\PEN_{n\times n}^{ss}$ is a sum of the part corresponding to the ${\cal K}$ blocks, { namely} $r(2r-1)$, the part corresponding to the interaction between the ${\cal K}$ blocks and ${\cal M}$ blocks, { namely} $2r(n-2w)$, and the part corresponding to the ${\cal M}$ blocks, { namely} $\sum_{i<j} (2 \max \{m_i,m_j \} + \varepsilon_{ij})$, where $m_i$ and $m_j$ are the indices of the ${\cal M}$ blocks (either $\alpha$ or $\alpha +1$), $\varepsilon_{ij} = 2$ if $m_i=m_j$ and $\varepsilon_{ij} = 1$ otherwise, resulting in:
\begin{equation*} \label{codimcomp}
\begin{aligned}
&{ \cod_{\PEN} } \orb^c ({\cal W}) = (n-2w-s)(n-2w-s-1)(2\alpha +2)/2  \\ %\frac{1}{2}
&\phantom{a} + s(s-1)(2(\alpha+1) +2)/2  + s(n-2w-s)(2(\alpha+1) +1) + 2r(n-2w) + r(2r-1) \\
&\phantom{a} = (n-2w-s)(n-2w-s-1)(\alpha +1) + s(s-1)(\alpha+1)  + s(s-1) \\
&\phantom{a} + 2s(n-2w-s)(\alpha+1) + s(n-2w-s) + 2r(n-2w) + r(2r-1)\\
&\phantom{a} = (\alpha +1) ((n-2w-s)(n-2w-1) + s(n-2w-1))  \\
&\phantom{a} + s(n-2w-1) + 2r(n-2w) + r(2r-1)\\
&\phantom{a} =(n-2w)(n-2w-1)(\alpha +1)+s(n-2w-1)  + 2r(n-2w) + r(2r-1)\\
&\phantom{a} =(n-2w-1)((n-2w)\alpha +s + n-2w)  + 2r(n-2w) + r(2r-1) \\
&\phantom{a} =(n-2w-1)(n-w-r) + 2r(n-2w) + r(2r-1).
\end{aligned}
\end{equation*}
Substituting $n=m(d+1)$ and $w = (md + 2r)/2$ in { the previous identity} we obtain:
\begin{align*}
\cod_{\GSYL} \orb({\cal F}_{\widetilde{P}}) & = (n - 2w - 1)(n-w-r) +2r(n-2w) + r(2r-1) \\
            & = (m(d+1) - 2r -md - 1)(m(d+1) - (2r + md)/2 -r) \\
            &+2r(m(d+1) -2r -md) + r(2r-1) \\
            & = (m - 2r - 1)(md+2m-4r)/2 +2mr - { 2 r^2} - r\\
             & = (m - 2r - 1)(md+m-2r)/2 +m(m-1)/2 .
\end{align*}
Taking into account equation \eqref{codequalities}, we obtain:
\begin{align*} { \cod_{\POL} \orb(P)} & = 
\cod_{\GSYL^0} \orb({\cal F}_{\widetilde{P}}) = \cod_{\GSYL} \orb({\cal F}_{\widetilde{P}}) - m(m-1)/2 \\
&= (m - 2r - 1)(md+m-2r)/2,
\end{align*}
{ which coincides with the codimension found in \cite[p. 15]{DmDo18} for the odd-grade case.}

\section{Conclusions and future work} \label{sect.conclusions}
This paper closes the knowledge gap for skew-symmetric matrix polynomials  with bounded rank and fixed grade. Namely, it provides the generic complete eigenstructure for skew-symmetric matrix polynomials of rank at most $2r$ and even grade $d$. It is also the first paper that tackles such problem for the structured matrix polynomials of even grade. Moreover, the obtained formulas (for the generic eigenstructures and codimensions) are exactly the same as in the case of odd $d$ obtained in \cite{DmDo18}, showing that the restriction to the odd case was motivated  by a need of using the linearization techniques.

The extensions of this result to another classes of structured matrix polynomials of even grade, such as symmetric, symmetric/skew-symmetric, or palindromic, is a part of our future work.

\section*{Acknowledgements}

% The authors thank to the anonymous referee for the helpful remarks and suggestions.

The work of A. Dmytryshyn was supported by the Swedish Research Council (VR) grant 2021-05393.
The work of F. De Ter\'an and F. M. Dopico has been partially funded by the Agencia Estatal de Investigaci\'on of Spain through grants PID2019-106362GB-I00 MCIN/
AEI/ 10.13039/ 501100011033/ and RED2022-134176-T, and by the Madrid Government (Comunidad de Madrid-Spain) under the Multiannual Agreement with UC3M in the line of Excellence of University Professors (EPUC3M23), and in the context of the V PRICIT (Regional Programme of Research and Technological Innovation).

{\footnotesize
\bibliographystyle{abbrv}
\bibliography{library}
}
\end{document}